\documentclass[12pt, a4paper]{article}
\usepackage{latexsym,amssymb,amsfonts,amsmath,epsfig,tabularx}

\usepackage{graphics}

\usepackage[hidelinks=true]{hyperref}

\usepackage{verbatim}

\usepackage{enumitem} 

\usepackage[usenames]{xcolor}
\newif\ifmarkauthors
\markauthorstrue
\ifmarkauthors
  \definecolor{darkred}{RGB}{139,0,0}
  \definecolor{darkgreen}{RGB}{0,100,0}
  \definecolor{darkmagenta}{RGB}{139,0,139}
  \definecolor{darkorange}{RGB}{190,70,20}

  \def\cbdelete[#1]{}
\else

  \def\cbdelete[#1]{}
\fi

\newtheorem{theorem}{Theorem}
\newtheorem{remark}[theorem]{Remark}

\newtheorem{proposition}[theorem]{Proposition}

\newtheorem{example}[theorem]{Example}

\def\rho{\varrho}

\newcommand{\esup}{\operatornamewithlimits{ess\,sup}}
\newcommand{\einf}{\operatornamewithlimits{ess\,inf}}

\newenvironment{proof}{\begin{trivlist}\item[\hskip\labelsep{\bf Proof.}]}{$\hfill\Box$\end{trivlist}}

\newcommand{\rd}{{\mathrm{d}}}

\newcommand{\calE}{{\mathcal{E}}}

\newcommand{\expect}{{\mathbb{E}}}
\newcommand{\bbR}{{\mathbb{R}}}
\newcommand{\one}{{\boldsymbol{1}}}
\newcommand{\mask}[1]{}

\title{On alternative quantization for doubly weighted\\ 
approximation and integration over unbounded domains}

\author
{P. Kritzer\thanks{P. Kritzer is supported by the Austrian
Science Fund (FWF): Project F5506-N26, which is a part of the Special Research Program
"Quasi-Monte Carlo Methods: Theory and Applications".}, 
F. Pillichshammer\thanks{F. Pillichshammer is supported by the Austrian Science Fund (FWF): 
Project F5509-N26, which is a part of the Special Research Program 
"Quasi-Monte Carlo Methods: Theory and Applications".}, 
L. Plaskota\thanks{L. Plaskota is supported by the National Science Centre, Poland:  
Project 2017/25/B/ST1/00945.}, 
G. W. Wasilkowski}
\date{\today}

\begin{document}
\maketitle

\begin{abstract}
It is known that for a $\rho$-weighted $L_q$-approximation of single variable functions $f$ 
with the $r$th derivatives in a $\psi$-weighted $L_p$ space, the minimal error of approximations 
that use $n$ samples of $f$ is proportional to 
$\|\omega^{1/\alpha}\|_{L_1}^\alpha\|f^{(r)}\psi\|_{L_p}n^{-r+(1/p-1/q)_+},$ where 
$\omega=\rho/\psi$ and $\alpha=r-1/p+1/q.$ Moreover, the optimal sample points are 
determined by quantiles of $\omega^{1/\alpha}.$ In this paper, we show how the error of best 
approximations 
changes when the sample points are determined by a quantizer $\kappa$ other 
than $\omega.$ Our results can be applied in situations when an alternative quantizer  has to be 
used because $\omega$ is not known exactly or is too complicated to handle computationally.
The results for $q=1$ are also applicable to $\rho$-weighted integration over unbounded domains. 
\end{abstract}

\centerline{\begin{minipage}[hc]{130mm}{
}
\end{minipage}}

\section{Introduction}

In various applications, continuous objects (signals, images, etc.) are represented (or approximated) by 
their discrete counterparts. That is, we deal with \emph{quantization}. From a pure mathematics point of view, 
quantization often leads to approximating functions from a given space by step functions or, 
more generally, by (quasi-)interpolating piecewise polynomials of certain degree. 
Then it is important to know which quantizer should be used, or how to select $n$ break points (knots) to make 
the error of approximation as small as possible.  

It is well known that for $L_q$ approximation on 
a compact interval $D=[a,b]$ in the space $F^r_p(D)$ of real-valued functions $f$ such that $f^{(r)}\in L_p(D),$ 
the choice of an optimal quantizer is not a big issue, since equidistant knots lead to approximations with optimal 
$L_q$ error
\begin{equation}\label{noweights}
    c\,(b-a)^\alpha \|f^{(r)}\|_{L_q}n^{-r+(1/p-1/q)_+}\qquad\quad
    \mbox{with}\quad  \alpha:=r-\frac1p+\frac1q,
\end{equation}
where $c$ depends only on $r$, $p$, and $q$, and where $x_+:=\max(x,0)$.
The problem becomes more complicated if we switch to weighted approximation on unbounded domains. 
A generalization of \eqref{noweights} to this case was given in \cite{KuPlWa16}, and it reads as follows. 
Assume for simplicity that the domain $D=\mathbb R_+:=[0,+\infty).$ Let $\psi,\rho:D\to(0,+\infty)$ be two 
positive and integrable \emph{weight} functions. For a positive integer $r$ and $1\le p,q\le+\infty,$ consider 
the $\rho$-weighted $L_q$ approximation in the linear space $F^r_{p,\psi}(D)$ of functions $f:D\to\mathbb R$ 
with absolutely (locally) continuous $(r-1)$st derivative and such that the $\psi$-weighted $L_p$ norm 
of $f^{(r)}$ is finite, i.e., $\|f^{(r)}\psi\|_{L_p}<+\infty.$
Note that the spaces $F^r_{p,\psi}(D)$ have been introduced in \cite{WaWo2000}, and 
the role of $\psi$ is to 
moderate their size. 

Denote 
\begin{equation}\label{omeg}
 \omega:=\frac\rho\psi, 
\end{equation}
and suppose that $\omega$ and $\psi$ are nonincreasing on $D,$ and that 
\begin{equation}\label{omega}
     \|\omega^{1/\alpha}\|_{L_1}:=\int_D\omega^{1/\alpha}(x)\,\rd x<+\infty.
\end{equation}
It was shown in \cite[Theorem~1]{KuPlWa16} that then one can construct approximations using $n$ knots with 
$\rho$-weighted $L_q$ error at most
$$ c_1\,\|\omega^{1/\alpha}\|_{L_1}^\alpha\,\|f^{(r)}\psi\|_{L_p}
n^{-r+(1/p-1/q)_+}.
$$
This means that if \eqref{omega} holds true, then the upper bound on the worst-case error is proportional to 
$\|\omega^{1/\alpha}\|_{L_1}^\alpha\,n^{-r+(1/p-1/q)_+}$. The convergence rate $n^{-r+(1/p-1/q)_+}$ 
is optimal and a corresponding lower bound implies that  if \eqref{omega} is not satisfied then 
the rate $n^{-r+(1/p-1/q)_+}$ cannot be reached (see \cite[Theorem~3]{KuPlWa16}).

The optimal knots
$$0=x_0^*<x_1^*<\ldots<x_{n-1}^*<x_n^*=+\infty
$$
are determined by quantiles of $\omega^{1/\alpha},$ to be more precise,
\begin{equation}\label{knots} 
  \int_0^{x_i^*}\omega^{1/\alpha}(t)\,\rd t=\frac in\,\|\omega^{1/\alpha}\|_{L_1}. 
\end{equation}
In order to use the optimal quantizer \eqref{knots} one has to know $\omega$; otherwise
he has to rely on some approximations of $\omega.$ Moreover, even if $\omega$ is known, 
it may be a complicated and/or non-monotonic function and therefore difficult to handle computationally. 
Driven by this motivation, the purpose of the present paper is to generalize the results of \cite{KuPlWa16} 
even further to see how the quality of best approximations will change if the optimal quantizer $\omega$ 
is replaced in \eqref{knots} by another quantizer $\kappa.$ 

A general answer to the aforementioned question is given in Theorems \ref{main1} and \ref{main2} of 
Section \ref{sect:second}. They show, respectively, tight (up to a constant) upper and lower bounds for the error when 
a quantizer $\kappa$ with $\|\kappa^{1/\alpha}\|_{L_1}<+\infty$ instead of $\omega$ is used to determine the knots. 
To be more specific, define
\begin{equation}\label{Edef1} \mathcal E_p^q(\omega,\kappa)\,=\,
\left\|\frac{\omega}{\kappa}\right\|_{L_\infty}\qquad\mbox{for $p\le q$},
\end{equation}
and 
\begin{equation}\label{Edef2} \mathcal E_p^q(\omega,\kappa)\,=\,\left(\int_D
\frac{\kappa^{1/\alpha}(x)}{\|\kappa^{1/\alpha}\|_{L_1}}
\bigg(\frac{\omega(x)}{\kappa(x)}\bigg)^{\frac{1}{1/q-1/p}}
\rd x\right)^{1/q-1/p}\qquad\mbox{for $p \ge q$.}
\end{equation} 
(Note that \eqref{Edef1} and \eqref{Edef2} are consistent for $p=q$.)
If $\mathcal E_p^q(\omega,\kappa)<+\infty$ then the best achievable error is proportional to 
$$ \|\kappa^{1/\alpha}\|_{L_1}^\alpha\,\mathcal E_p^q(\omega,\kappa)\,\|f^{(r)}\psi\|_{L_p}n^{-r+(1/p-1/q)_+}.$$
This means, in particular, that for the error to behave as $n^{-r+(1/p-1/q)_+}$ it is sufficient (but not necessary) that $\kappa(x)$ 
decreases no faster than $\omega(x)$ as $|x|\to+\infty.$ 
For instance, if the optimal quantizer is Gaussian, $\omega(x)=\exp(-x^2/2),$ 
then the optimal rate is still preserved if its exponential substitute $\kappa(x)=\exp(-a|x|)$ with arbitrary $a>0$ is used. 
It also shows that, in case $\omega$ is not exactly known, it is much safer to overestimate than underestimate it, 
see also Example~\ref{example1}.

The use of a quantizer $\kappa$ as above results in approximations that are worse than the optimal approximations by the factor of 
$$ \mathrm{FCTR}(p,q,\omega,\kappa)\,=\,\frac{\|\kappa^{1/\alpha}\|_{L_1}^\alpha}{\|\omega^{1/\alpha}\|_{L_1}^\alpha}
    \,\mathcal E_p^q(\omega,\kappa)\,\ge\,1. $$
In Section \ref{sec:third}, we calculate the exact values of this factor for various combinations of weights 
$\rho,$ $\psi$, and $\kappa$, including: Gaussian, exponential, log-normal, logistic, and $t$-Student. 
It turns out that in many cases $\mathrm{FCTR}(p,q,\omega,\kappa)$ 
is quite small, so that the loss in accuracy of approximation is well compensated by simplification of the weights.

The results for $q=1$ are also applicable for problems of
approximating $\rho$-weighted integrals 
\[\int_D f(x)\,\rho(x)\,\rd x\quad\mbox{for}\quad f\,\in\,
F^r_{p,\psi}(D).
\]  
More precisely, the worst case errors of quadratures that are integrals of the
corresponding piecewise interpolation polynomials approximating
functions $f\in F^r_{p,\psi}(D)$ are the same as the errors
for the $\rho$-weighted $L_1(D)$ approximations. Hence 
their errors, proportional to $n^{-r}$, are (modulo a constant) the best
possible among all quadratures. 
These results are especially important for unbounded
domains, e.g., $D=\bbR_+$ or $D=\bbR$.
For such domains, the integrals are often
approximated by Gauss-Laguerre rules and Gauss-Hermite rules, respectively, 
see, e.g., \cite{DaRa1984,Hil1954,StSe1966}; however,
their efficiency requires smooth integrands and the results are
asymptotic. Moreover, it is not clear which Gaussian rules should be used
when $\psi$ is not a constant function.
But, even for $\psi\equiv1$, it is likely that the worst case errors
(with respect to $F^r_{p,\psi}$) of Gaussian rules are much larger
than $O(n^{-r})$, since the Weierstrass theorem holds only for compact $D$. 
A very interesting extension of Gaussian rules to functions with
singularities has been proposed in \cite{GrOe14}. However, 
the results of \cite{GrOe14} are also asymptotic and it is not
clear how 
the proposed rules  behave for functions
from spaces $F^r_{p,\psi}$. 
In the present paper,
we deal with functions of bounded smoothness ($r<+\infty$) and 
provide worst-case error bounds that are minimal. 
We stress here that the regularity degree $r$ is a fixed
but arbitrary positive integer. 
The paper \cite{KrPiPlWa2018}
proposes a different approach to the weighted integration over
unbounded domains; however, it is restricted to regularity $r=1$ only.

The paper is organized as follows. In the following section, we present 
ideas and results  
about alternative quantizers. The main results are Theorems~\ref{main1} and \ref{main2}. 
In Section~\ref{sec:third}, we apply our results to some specific cases 
for which 
numerical 
values 
of ${\rm FCTR}(p,q,\omega,\kappa)$ are calculated.  

\section{Optimal versus alternative quantizers}\label{sect:second}

We consider $\rho$-weighted $L_q$ approximation in the space $F_{p,\psi}^r(D)$ as defined in the introduction; however, 
in contrast to \cite{KuPlWa16}, we do not assume that the weights $\psi$ and $\omega$  are nonincreasing. Although the results of this paper pertain to domains $D$ being 
an arbitrary interval, to begin with we assume that 
$$D=\mathbb R_+.$$ 
We will explain later what happens in the general case including $D=\mathbb R.$ 

\smallskip
Let the knots $0=x_0<\ldots<x_n=+\infty$ be determined by a nonincreasing 
function (quantizer) $\kappa:D\to(0,+\infty)$ satisfying $\|\kappa^{1/\alpha}\|_{L_1}<+\infty,$ i.e.,
\begin{equation}\label{knots1}
   \int_0^{x_i}\kappa^{1/\alpha}(t)\,\rd t=\frac in\,\|\kappa^{1/\alpha}\|_{L_1}\quad\mbox{with}\quad \alpha\,=\,r-\frac1p+\frac1q. 
\end{equation}

Let $\mathcal T_nf$ be a piecewise Taylor approximation of $f\in F^r_{p,\psi}(D)$ with break-points \eqref{knots1},  
$$\mathcal T_nf(x) = \sum_{i=1}^{n}\one_{[x_{i-1},x_i)}(x)\sum_{k=0}^{r-1}\frac{f^{(k)}(x_{i-1})}{k!}(x-x_{i-1})^k.$$

We remind the reader of the definition of the quantity $\mathcal E_p^q(\omega,\kappa)$ in \eqref{Edef1} and \eqref{Edef2}, which will be of importance 
in the following theorem.

\begin{theorem}\label{main1}
Suppose that 
$$  \mathcal E_p^q(\omega,\kappa)<+\infty. $$ 
Then for every $f\in F_{p,\psi}^q(D)$ we have 
\begin{equation}\label{eqthm}
     \|(f-\mathcal T_nf)\rho\|_{L_q}\,\le\,c_1\,\|\kappa^{1/\alpha}\|_{L_1}^{{\alpha}}
     \,\mathcal E_p^q(\omega,\kappa)\,\|f^{(r)}\psi\|_{L_p}\,n^{-r+(1/p-1/q)_+},
\end{equation}
where $$  c_1\,=\,\frac1{(r-1)!\,((r-1)p^*+1)^{1/p^*}}. $$ 
\end{theorem}

\begin{proof}
We proceed as in the proof of \cite[Theorem~1]{KuPlWa16} to get that for $x\in[x_{i-1},x_i)$ 
\begin{eqnarray*}
     \rho(x)|f(x)-\mathcal T_nf(x)|  &=& \rho(x)\bigg|\int_{x_{i-1}}^{x_i}f^{(r)}(t)\frac{(x-t)_+^{r-1}}{(r-1)!}\rd t\bigg| \\
      &\le & c_1\,\frac{\omega(x)}{\kappa(x)}\,
       \bigg(\int_{x_{i-1}}^{x_i}|f^{(r)}(t)\psi(t)|^p\rd t\bigg)^{1/p}\kappa(x)(x-x_{i-1})^{r-1/p}. 
\end{eqnarray*}
Since (cf. \cite[p.36]{KuPlWa16})
$$ \kappa(x)(x-x_i)^{r-1/p}\le 
    (\kappa^{1/\alpha}(x))^{1/q}\bigg(\frac{\|\kappa^{1/\alpha}\|_{L_1}}{n}\bigg)^{r-1/p}, $$
the error is upper bounded as follows:
\begin{align}\label{est1}
  \lefteqn{ \|(f-\mathcal T_nf)\rho\|_{L_q} =
  \bigg(\sum_{i=1}^{n}\int_{x_{i-1}}^{x_i}\rho^q(x)|f(x)-\mathcal T_nf(x)|^q\rd x\bigg)^{1/q}  }  \nonumber\\
 & \le c_1\bigg(\frac{\|\kappa^{1/\alpha}\|_{L_1}}{n}\bigg)^{r-1/p}\left(\sum_{i=1}^{n}
     \bigg(\int_{x_{i-1}}^{x_i}\kappa^{1/\alpha}(x)\bigg(\frac{\omega(x)}{\kappa(x)}\bigg)^q\rd x\bigg)
     \bigg(\int_{x_{i-1}}^{x_i}|f^{(r)}(t)\psi(t)|^p\rd t\bigg)^{q/p}\right)^{1/q}.
\end{align}
Now we maximize the right hand side of \eqref{est1} subject to 
$$\|f^{(r)}\psi\|_{L_p}^p=\sum_{i=1}^n\int_{x_{i-1}}^{x_i}|f^{(r)}(t)\psi(t)|^p\rd t=1.$$
After the substitution
$$ A_i := \int_{x_{i-1}}^{x_i}\kappa^{1/\alpha}(x)\bigg(\frac{\omega(x)}{\kappa(x)}\bigg)^q\rd x,
 \qquad B_i := \bigg(\int_{x_{i-1}}^{x_i}|f^{(r)}(t)\psi(t)|^p\rd t\bigg)^{q/p},
$$
this is equivalent to $$\mbox{maximizing\quad
$\sum_{i=1}^{n}A_iB_i$\quad subject to\quad $\sum_{i=1}^{n}B_i^{p/q}=1$.}$$
We have two cases:

For $p\le q$, we set  $i^*=\mathrm{arg}\max_{1\le i\le n}A_i,$ and use Jensen's inequality to obtain $$\sum_{i=1}^{n}A_iB_i \le A_{i^*} \sum_{i=1}^{n}B_i \le A_{i^*} \left( \sum_{i=1}^{n}B_i^{p/q}\right)^{q/p}=A_{i^*}.$$ Hence the maximum equals $A_{i^*}$ and it is attained  at $B_i^*=1$ for $i=i^*,$ and $B_i^*=0$ otherwise. In this case, the maximum is upper bounded by 
$\|\omega/\kappa\|_{L_\infty}^q\|\kappa^{1/\alpha}\|_{L_1}/n,$ which means that 
$$ \|(f-\mathcal T_nf)\rho\|_{L_q}\,\le\, c_1\bigg(\frac{\|\kappa^{1/\alpha}\|_{L_1}}{n}\bigg)^\alpha
      \left\|\frac\omega\kappa\right\|_{L_\infty}\|f^{(r)}\psi\|_{L_p}. $$

For $p>q$ we use the method of Lagrange multipliers and find this way that the maximum equals 
$$ \left(\sum_{i=1}^{n}A_i^\frac{1}{1-q/p}\right)^{1-q/p}\,=\,
    \left( \sum_{i=1}^{n}\bigg(\int_{x_{i-1}}^{x_i}\kappa^{1/\alpha}(x)
      \left(\frac{\omega(x)}{\kappa(x)}\right)^q\rd x\bigg)^\frac{1}{1-q/p} \right)^{1-q/p}, $$
and is attained at 
$$  B_i^*=\left(\frac{A_i^\frac{1}{1-q/p}}{\sum_{j=1}^n A_j^\frac{1}{1-q/p}}\right)^{q/p},\qquad 1\le i\le n. $$
Since $1/(1-q/p)>1,$ by the probabilistic version of Jensen's inequality with density 
$n\,\kappa^{1/\alpha}/\|\kappa^{1/\alpha}\|_{L_1},$ we have 
$$ \bigg(\int_{x_{i-1}}^{x_i}\kappa^{1/\alpha}(x)\left(\frac{\omega(x)}{\kappa(x)}\right)^q\rd x\bigg)^\frac{1}{1-q/p} 
     \le \bigg(\frac{\|\kappa^{1/\alpha}\|_{L_1}}{n}\bigg)^\frac{1}{p/q-1}
     \int_{x_{i-1}}^{x_i}\kappa^{1/\alpha}(x)\left(\frac{\omega(x)}{\kappa(x)}\right)^\frac{1}{1/q-1/p}\rd x. $$
This implies that
$$ \left(\sum_{i=1}^{n}A_i^\frac{1}{1-q/p}\right)^{1-q/p}\,\le\,
     \bigg(\frac{\|\kappa^{1/\alpha}\|_{L_1}}{n}\bigg)^{q/p}
   \left(\int_0^{+\infty}\kappa^{1/\alpha}(x)\left(\frac{\omega(x)}{\kappa(x)}\right)^\frac{1}{1/q-1/p}\rd x\right)^{1-q/p},
$$
and finally
$$ \|(f-\mathcal T_nf)\rho\|_{L_q}\,\le\,c_1\bigg(\frac{\|\kappa^{1/\alpha}\|_{L_1}}{n}\bigg)^r
      \left(\int_0^{+\infty}\kappa^{1/\alpha}(x)\left(\frac{\omega(x)}{\kappa(x)}\right)^\frac{1}{1/q-1/p}
       \rd x\right)^{1/q-1/p}\|f^{(r)}\psi\|_{L_p}, 
$$ 
as claimed since $1/q-1/p=\alpha-r$.
\end{proof}

\begin{remark}\rm
If derivatives of $f$ are difficult to compute or to sample,
a piecewise Lagrange interpolation $\mathcal L_n$ can be used,
as in \cite{KuPlWa16}. Then the result is slightly weaker than that of the present
Theorem~\ref{main1};
namely (cf. \cite[Theorem~2]{KuPlWa16}), there exists $c_1'>0$ depending only on $p$, $q,$ and $r,$ such that
$$ \limsup_{n\to\infty}\sup_{f\in F_{p,\psi}^r(D)}\frac{\|(f-\mathcal L_nf)\rho\|_{L_q}}{\|f^{(r)}\psi\|_{L_p}}\,
        n^{r+(1/p-1/q)_+}\,\le\,c_1'\,\|\kappa^{1/\alpha}\|_{L_1}^{\alpha}\,\mathcal E_p^q(\omega,\kappa). $$
 
\end{remark}

\medskip
We now show that the error estimate of Theorem \ref{main1} cannot be improved.

\begin{theorem}\label{main2}
There exists $c_2>0$ depending only on $p,$ $q,$ and $r$ with the following property. For any approximation $\mathcal A_n$ 
that uses only information about function values and/or its derivatives (up to order $r-1$) at the knots $x_0,\ldots,x_n$ 
given by \eqref{knots1}, we have
\begin{equation}\label{eqth2}
  \liminf_{n\to\infty}\sup_{f\in F_{p,\psi}^r(D)}\frac{\|(f-\mathcal A_nf)\rho\|_{L_q}}{\|f^{(r)}\psi\|_{L_p}}\,
        n^{r-(1/p-1/q)_+}\,\ge\,c_2\,\|\kappa^{1/\alpha}\|_{L_1}^{\alpha}\mathcal E_p^q(\omega,\kappa). 
\end{equation}
\end{theorem}

\begin{proof}
We fix $n$ and consider first the weighted $L_q$ approximation on $[0,x_{n-1})$ assuming that in this interval the weights are step functions 
with break points $x_i$ given by \eqref{knots1}. Let 
$\psi_i,$ $\rho_i,$ $\omega_i=\rho_i/\psi_i,$ and $\kappa_i$ 
be correspondingly the values of 
$\psi,$ $\rho,$ $\omega,$ and $\kappa$ on successive intervals $[x_{i-1},x_i).$ Then we clearly have that
$(x_i-x_{i-1})\kappa_i^{1/\alpha}=\|\kappa^{1/\alpha}\|_{L_1(0,x_{n-1})}/(n-1).$ 

For simplicity, we write $I_i:=(x_{i-1},x_i).$ Let $f_i,$ $1\le i\le n-1,$ be functions supported on $I_i,$ such that 
$f_i^{(j)}(x_{i-1})=0=f_i^{(j)}(x_i)$ for $0\le j\le r-1,$ and 
\begin{equation}\label{c2}
    \|f_i\|_{L_q(I_i)}\ge c_2(x_i-x_{i-1})^\alpha\|f_i^{(r)}\|_{L_p(I_i)}. 
\end{equation}
We also normalize $f_i$ so that $\|f_i^{(r)}\|_{L_p(I_i)}=1/\psi_i.$ We stress that a positive $c_2$ in \eqref{c2} exists and 
depends only on $r,$ $p,$ and $q.$ 

Since all $f_i^{(j)}$ nullify at the knots $x_k,$ the `sup' (worst case error) in \eqref{eqth2} is bounded from below by
$$ \mathrm{Sup}(n)\,:=\,\sup\bigg\{\|f\rho\|_{L_q}:\;f=\sum_{i=1}^{n-1}\beta_if_i,\;\sum_{i=1}^{n-1}|\beta_i|^p=1\bigg\}, $$
where we used the fact that $\|f^{(r)}\psi\|_{L_p}=\left(\sum_{i=1}^{n-1}|\beta_i|^p\right)^{1/p}.$ For such $f$ we have
\begin{eqnarray*}
  \|f\rho\|_{L_q} &=& \bigg(\sum_{i=1}^{n-1}\beta_i^q\|f_i\rho\|_{L_q(I_i)}^q\bigg)^{1/q}\,=\,
    \bigg(\sum_{i=1}^{n-1}\left(|\beta_i|\rho_i\|f_i\|_{L_q(I_i)}\right)^q\bigg)^{1/q} \\
    &\ge & c_2\,\bigg(\sum_{i=1}^{n-1}\left(|\beta_i|\rho_i(x_i-x_{i-1})^\alpha\|f_i^{(r)}\|_{L_p(I_i)}\right)^q\bigg)^{1/q} \\
    &=& c_2\,\bigg(\sum_{i=1}^{n-1}\left(|\beta_i|\, \frac{\omega_i}{\kappa_i} \ \kappa_i(x_i-x_{i-1})^\alpha\right)^q\bigg)^{1/q} \\
    &=& c_2\,\bigg(\frac{\|\kappa^{1/\alpha}\|_{L_1}}{n-1}\bigg)^\alpha
                \bigg(\sum_{i=1}^{n-1}|\beta_i|^q\left(\frac{\omega_i}{\kappa_i}\right)^q\bigg)^{1/q}.
\end{eqnarray*}
Thus we arrive at a maximization problem that we already had in the
proof of Theorem \ref{main1}. 

For $p\le q$ we have 
\begin{eqnarray*}
     \mathrm{Sup}(n) &=&
     c_2\,\bigg(\frac{\|\kappa^{1/\alpha}\|_{L_1}}{n-1}\bigg)^\alpha
     \max_{1\le i\le n-1} \frac{\omega_i}{\kappa_i} \,=\,
     c_2\,\bigg(\frac{\|\kappa^{1/\alpha}\|_{L_1}}{n-1}\bigg)^\alpha
      \esup_{0\le x<x_{n-1}}\frac{\omega(x)}{\kappa(x)},
\end{eqnarray*}
while for $p>q$ we have
\begin{eqnarray*}
 \mathrm{Sup}(n) &=&
 c_2\,\,\bigg(\frac{\|\kappa^{1/\alpha}\|_{L_1}}{n-1}\bigg)^\alpha
 \left(\sum_{i=1}^{n-1}\bigg(\frac{\omega_i}{\kappa_i}
 \bigg)^\frac{1}{{\alpha-r}}\right)^{{\alpha-r}} \\
 &=& c_2\,\bigg(\frac{\|\kappa^{1/\alpha}\|_{L_1}}{n-1}\bigg)^r
 \left(\sum_{i=1}^{n-1}\bigg(\frac{\|\kappa^{1/\alpha}\|_{L_1}}{n-1}\bigg)
  \bigg(\frac{\omega_i}{\kappa_i}\bigg)^\frac{1}{{\alpha-r}}
  \right)^{{\alpha-r}} \\
      &=& c_2\,\bigg(\frac{\|\kappa^{1/\alpha}\|_{L_1}}{n-1}\bigg)^r
       \left(\int_0^{x_{n-1}}\kappa^{1/\alpha}(x)\bigg(\frac{\omega(x)}{\kappa(x)}\bigg)^\frac{1}{{\alpha-r}}\rd x\right)^{{\alpha-r}},
\end{eqnarray*}
as claimed. 

For arbitrary weights, we replace $\psi,$ $\rho,$ and $\kappa$ with the corresponding step functions with 
$$\psi_i=\esup_{x\in (x_{i-1},x_i)}\psi(x),\quad\rho_i=\einf_{x\in (x_{i-1},x_i)}\rho(x),\quad 
   \kappa_i=\bigg(\frac{\|\kappa^{1/\alpha}\|_{L_1}}{n(x_i-x_{i-1})}\bigg)^\alpha,\qquad 1\le i\le n-1, $$
and go with $n$ to $+\infty.$
\end{proof}

We now comment on what happens when the domain is different from $\mathbb R_+.$ It is clear that 
Theorems \ref{main1} and \ref{main2} remain valid for $D$ being a compact interval, say $D=[0,c]$ with $c<+\infty.$ 
Consider $$D=\mathbb R.$$ In this case, we assume that $\kappa$ is nonincreasing on $[0,+\infty)$ and nondecreasing 
on $(-\infty,0].$ We have $2n+1$ knots $x_i,$ which are determined by the condition
\begin{equation}\label{condxigen}
\int_0^{x_i}\kappa^{1/\alpha}(t)\,\rd t\,=\,\frac{i}{2 n}\,\|\kappa^{1/\alpha}\|_{L_1(\mathbb R)},\qquad |i|\le n
\end{equation}
(where $\int_0^{-a}=-\int_a^0$). Note that \eqref{condxigen} automatically implies $x_0=0$.
The piecewise Taylor approximation is also correspondingly defined for negative 
arguments. With these modifications, the corresponding Theorems \ref{main1} and \ref{main2} have literally 
the same formulation for $D=\mathbb R$ and for $D=\mathbb R_+.$

\medskip
Observe that the error estimates of Theorems \ref{main1} and \ref{main2} for arbitrary $\kappa$ differ 
from the error for optimal $\kappa=\omega$ by the factor
$$ \mathrm{FCTR}(p,q,\omega,\kappa)\,:=\,
     \frac{\|\kappa^{1/\alpha}\|_{L_1}^\alpha}{\|\omega^{1/\alpha}\|_{L_1}^\alpha}\,
     \mathcal E_p^q(\omega,\kappa).$$
From this definition it is clear that for any $s,t>0$ we have 
$$\mathrm{FCTR}(p,q,s\,\omega,t\,\kappa)=\mathrm{FCTR}(p,q,\omega,\kappa).$$ 
This quantity satisfies the following estimates.

\begin{proposition}
We have 
\begin{equation}\label{fctr} 
   1\,=\,\mathrm{FCTR}(p,q,\omega,\omega)\,\le\,\mathrm{FCTR}(p,q,\omega,\kappa)\,\le\,
    \frac{\|\kappa^{1/\alpha}\|_{L_1}^\alpha}{\|\omega^{1/\alpha}\|_{L_1}^\alpha}
    \left\|\frac{\omega}{\kappa}\right\|_{L_\infty}.
\end{equation}
The rightmost inequality is actually an equality whenever $p \le q$.
\end{proposition}

\begin{proof}
Assume without loss of generality that 
$\|\kappa^{1/\alpha}\|_{L_1}=\|\omega^{1/\alpha}\|_{L_1}=1,$ so that 
$\mathrm{FCTR}(p,q,\omega,\kappa)=\mathcal E_p^q(\omega,\kappa).$ Then for any $p$ and $q$
$$1=\|\omega^{1/\alpha}\|_{L_1}^\alpha\le\|\kappa^{1/\alpha}\|_{L_1}^\alpha
    \left\|\frac{\omega^{1/\alpha}}{\kappa^{1/\alpha}}\right\|_{L_\infty}^\alpha=
    \left\|\frac{\omega}{\kappa}\right\|_{L_\infty}, $$
which equals $\mathcal E_p^q(\omega,\kappa)$ for $p\le q.$ For $p>q$ we have 
$(1/q-1/p)/\alpha=1-r/\alpha<1,$ so that we can use
Jensen's inequality to get
$$  \mathcal E_p^q(\omega,\kappa)=\left(\int_D\kappa^{1/\alpha}(x)
      \bigg(\frac{\omega^{1/\alpha}(x)}{\kappa^{1/\alpha}(x)}\bigg)^{\frac{\alpha}{{\alpha-r}}}\rd x
      \right)^{\left(\frac{{\alpha-r}}{\alpha}\right)\alpha}\ge \left(\int_D\kappa^{1/\alpha}(x)
      \bigg(\frac{\omega^{1/\alpha}(x)}{\kappa^{1/\alpha}(x)}\bigg)\rd x\right)^\alpha=1. $$
The remaining inequality $\mathcal E_p^q(\omega,\kappa)\le \left\|\frac{\omega}{\kappa}\right\|_{L_\infty}$
is obvious.
\end{proof}

Although the main idea of this paper is to replace $\omega$ by
another function $\kappa$ that is easier to handle, 
our results allow a further interesting observation that is
illustrated in the following example.

\begin{example}\label{example1} \rm
Let $D=\mathbb R,$ $$r=1,\qquad p=+\infty,\qquad q=1, $$ and the weights
$$ \rho(x)=\frac{1}{\sqrt{2\pi}}\exp\left(\frac{-x^2}{2}\right),\qquad \psi(x)=1. $$
Then $\alpha=2$ and $1/q-1/p=1,$ and $\omega(x)=\rho(x)$. Suppose that instead of $\omega$ we use 
$$  \kappa_\sigma(x)=\frac{1}{\sqrt{2\pi\sigma^2}}\exp\left(\frac{-x^2}{2\sigma^2}\right)
     \qquad\mbox{with}\qquad\sigma^2>0. $$ 

\medskip\noindent
Since $p>q,$ we have
$$ \mathrm{FCTR}(p,q,\omega,\kappa_\sigma)\,=\,\frac{\|\kappa_{\sigma}^{1/2}\|_{L_1}^2}{\|\omega^{1/2}\|_{L_1}^2}
     \int_\mathbb R \frac{\kappa_{\sigma}^{1/2}(x)}{\|\kappa_{\sigma}^{1/2}\|_{L_1}}\,\frac{\omega(x)}{\kappa_{\sigma}(x)}\,\rd x \\
     \,=\, 
     \left\{\begin{array}{rl} +\infty & \,\mbox{if}\quad \sigma^2\le 1/2, \\
                          \frac{\sigma^2}{\sqrt{2\sigma^2-1}} & \,\mbox{if}\quad\sigma^2>1/2. 
       \end{array}\right. 
$$ 
The graph of $\mathrm{FCTR}(p,q,\omega,\kappa_\sigma)$ is drawn in Fig.~\ref{fctr1}. It follows that it is 
safer to overestimate the actual variance $\sigma^2=1$ than to underestimate it.
\begin{figure}
\begin{center}
\includegraphics[width=13cm]{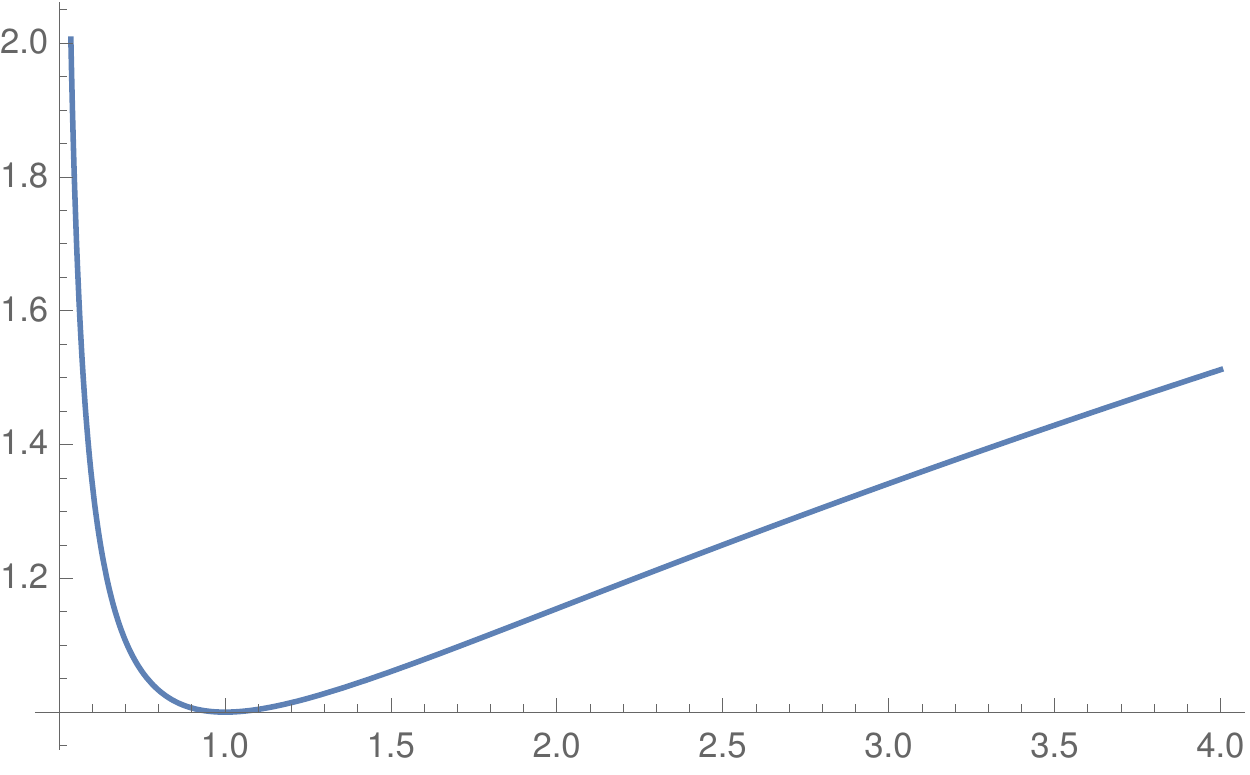}
\caption{Plot of $\mathrm{FCTR}(p,q,\omega,\kappa_\sigma)$ versus $\sigma^2$ from Example~\ref{example1}}
\label{fctr1}
\end{center}
\end{figure}
\end{example}

\section{Special cases}\label{sec:third}

Below we apply our results to specific 
weights $\rho,\psi$, and specific values of $p$ and $q$.

\subsection{Gaussian $\rho$ and $\psi$}
Consider $D=\bbR$, 
\[
\rho(x)\,=\,\frac1{\sigma\,\sqrt{2\,\pi}}\,
\exp\left(\frac{-x^2}{2\,\sigma^2}\right) \quad\mbox{and}\quad
  \psi(x)\,=\,\exp\left(\frac{-x^2}{2\,\lambda^2}\right) 
\]
for positive $\sigma$ and $\lambda$. Since
\[
\omega(x)\,=\,\frac1{\sigma\,\sqrt{2\,\pi}}\,
\exp\left(\frac{-x^2}2\,(\sigma^{-2}-\lambda^{-2})\right),
\]
for $\|\omega^{1/\alpha}\|_{L_1}<\infty$ we have to have 
$\lambda\,>\,\sigma$, and then
\[\|\omega^{1/\alpha}\|_{L_1}^\alpha\,=\,\frac1{\sigma\,\sqrt{2\,\pi}}\,
\left(\frac{\alpha\,2\,\pi}{\sigma^{-2}-\lambda^{-2}}\right)^{\alpha/2}.
\]
We propose using
\[\kappa(x)\,=\,\kappa_a(x)\,=\,\exp(-|x|\,a)\quad\mbox{for}\quad a\,>\,0.
\]
Then $\|\kappa_a^{1/\alpha}\|_{L_1(D)}=2 \alpha/a$ and the points $x_{-n},\dots,x_{n}$ satisfying \eqref{condxigen},
\[
\int_{0}^{x_i}\kappa_a^{1/\alpha}(t)\,\rd t\,=\,\frac{i}{2 n}\,
\int_{-\infty}^\infty\kappa_a^{1/\alpha}(t)\,\rd t\quad\mbox{for}
\quad |i|\le n,
\]
are given by
\begin{equation}\label{GG-quant}
    x_{i}\,=\,-x_{-i}\,=\,-\frac{\alpha}a\,
  \ln\left(1-\frac{i}{n}\right)\quad\mbox{for\ } 0 \le i \le n.
\end{equation}
In particular, we have
\[x_{-n}\,=\,-\infty,\quad x_0\,=\,0, \quad\mbox{and}\quad
x_n\,=\,\infty. 
\]

We now consider the two cases $p\le q$ and $p>q$ separately:

\subsubsection{Case of $p\le q$}
Clearly
\[\calE^q_p(\omega,\kappa_a)\,=\,
\left\|\frac{\omega}{\kappa_a}\right\|_{L_\infty(D)}\,=\,
\frac1{\sigma\,\sqrt{2\,\pi}}\,\exp\left(\frac{a^2}{2\,(\sigma^{-2}
-\lambda^{-2})}\right)
\]
and
\[
\min_{a>0}\|\kappa_a^{1/\alpha}\|_{L_1(D)}^\alpha\,
\left\|\frac\omega{\kappa_a}\right\|_{L_\infty}\quad
\mbox{is attained at}\quad 
a_*\,=\,\sqrt{\alpha\,\left(\frac{1}{\sigma^2}-\frac{1}{\lambda^2}\right)}.
\]
Hence, for $p \le q$ we have that
\[
\mbox{FCTR}(p,q,\omega,\kappa_{a_*})\,=\,
\left(\frac{2\,{\rm e}}{\pi}\right)^{\alpha/2}.
\]
Note that ${\rm FCTR}(p,q,\omega,\kappa_{a_*})$ does not depend 
on $\sigma$ and $\lambda$ (as long as $\lambda>\sigma$). 
For instance, we have the following rounded values:
$$
\begin{array}{c||c|c|c|c}
\alpha & 1 & 2 & 3 & 4 \\
\hline
{\rm FCTR}(p,q,\omega,\kappa_{a_*}) &  1.315 & 1.731 & 2.276 & 2.995
\end{array}
$$

\subsubsection{Case of $p>q$}
We have now 
\[ \calE_p^q(\omega,\kappa_a)\,=\,\left(\frac{a}\alpha
\right)^{{\alpha-r}}\,\frac1{\sigma\,\sqrt{2\,\pi}}\,A^{{\alpha-r}},
\]
where
\begin{eqnarray*}
  A&=&\int_0^\infty\exp\left(-\frac{x^2\,{(\sigma^{-2}-\lambda^{-2})}}
  {2\,{(\alpha-r)}}
   +\frac{a\,x\,{r}}{\alpha\,{(\alpha-r)}} \right)\,\rd x\\
   &=&\int_0^\infty\exp\left(-\frac{\sigma^{-2}-\lambda^{-2}}
   {2\,{(\alpha-r)}}\,
   \left(x-\frac{a\,{r}}{\alpha\,{(\sigma^{-2}-\lambda^{-2})}}
   \right)^2
   +\frac{a^2\,{r}^2}{2\,\alpha^2\,
   {(\alpha-r)\,(\sigma^{-2}-\lambda^{-2})}}\right)\,\rd x\\
   &=&\exp\left(\frac{a^2\,{r}^2}
   {2\,\alpha^2\,{(\alpha-r)\,(\sigma^{-2}-\lambda^{-2})}}\right)\,
   \int_{-\frac{a\,{r}}
     {\alpha\,{(\sigma^{-2}-\lambda^{-2})}}}^\infty
   \exp\left(-\frac{{(\sigma^{-2}-\lambda^{-2})}\,t^2}
    {2\,{(\alpha-r)}}\right)\rd t\\
    &=&\exp\left(\frac{a^2\,{r}^2}{2\,\alpha^2\,{(\alpha-r)\,
        \,(\sigma^{-2}-\lambda^{-2})}}\right)\,
   \sqrt{\frac{\pi\,{(\alpha-r)}}{2\,{(\sigma^{-2}-\lambda^{-2})}}}\,
   \left[1{+}{\rm erf}\left(\frac{a\,{r}}{\alpha\,
       \sqrt{2\,{(\alpha-r)\,(\sigma^{-2}-\lambda^{-2})}}}\right)\right],     
\end{eqnarray*}
where $\mathrm{erf}(z):=\frac{2}{\sqrt{\pi}}\int_0^z {\rm e}^{-t^2}\,\rd t$. This gives 
\begin{eqnarray*}
\calE_p^q(\omega,\kappa_a)&=&
\left(\frac{a^2\,\pi\,{(\alpha-r)}}{\alpha^2\,2\,
  {(\sigma^{-2}-\lambda^{-2})}}\right)^{{(\alpha-r)}/2}\,
\frac1{\sigma\,\sqrt{2\,\pi}}\,
\exp\left(\frac{a^2\,{r}^2}{\alpha^2\,2\,{(\sigma^{-2}-\lambda^{-2})}}
  \right)\\
 &&\qquad\times \left[1{+}{\rm erf}\left(\frac{a\,r}
 {\alpha\,\sqrt{2\,{(\alpha-r)\,(\sigma^{-2}-\lambda^{-2})}}}
\right)\right]^{{\alpha-r}}.
\end{eqnarray*}
Since $$\frac{\|\kappa_a^{1/\alpha}\|_{L_1(D)}^{\alpha}}{\|\omega^{1/\alpha}\|_{L_1(D)}^{\alpha}}=\sigma \sqrt{2 \pi} \left(\frac{2 \alpha (\sigma^{-2}-\lambda^{-2})}{\pi a^2}\right)^{\alpha/2}$$ we obtain 
\begin{eqnarray*}
\mbox{FCTR}(p,q,\omega,\kappa_{a}) & = &   \left(\frac{2 \alpha (\sigma^{-2}-\lambda^{-2})}{\pi a^2}\right)^{r/2} \left(\frac{\alpha-r}{\alpha}\right)^{(\alpha-r)/2}\,
\exp\left(\frac{a^2\,r^2}{2\,\alpha^2(\sigma^{-2}-\lambda^{-2})}\right)\\
& & \times \left[1+{\rm erf}\left(\frac{a\,r}{\alpha\,\sqrt{2\,(\alpha-r)\,(\sigma^{-2}-\lambda^{-2})}}\right)\right]^{\alpha-r}.
\end{eqnarray*}

We provide some numerical tests for $q=1$ and $p=2$ or
$p=\infty$. Then $\alpha=r+1/2$ or $\alpha=r+1$, respectively. 
Recall that results for $q=1$ are also applicable to the
$\rho$-integration problem.

For $r\in \{1,2\}$, $p\in \{2,\infty\},$ $\lambda=2$ and $\sigma=1$, we vary
$a$ and obtain the following rounded values:
$$
\begin{array}{l||c|c|c|c||c|l}
\ \ \ \ \ \ \ \ \ \ \ a & 1 & 2 & 3 & 4 & \\
\hline
{\rm FCTR}(2,1,\omega,\kappa_{a}) & 1.135   & 1.476 & 4.361 & 26.036 & r=1 &\\
{\rm FCTR}(2,1,\omega,\kappa_{a}) & 1.645   & 1.552 & 5.836 & 65.061 & r=2 & \raisebox{1.5ex}[-1.5ex] {$p=2$}\\
\hline
{\rm FCTR}(\infty,1,\omega,\kappa_{a}) & 1.172   & 1.179 & 1.979 & 4.920 & r=1 & \\
{\rm FCTR}(\infty,1,\omega,\kappa_{a}) & 1.733   & 1.269 & 2.617 & 11.826 & r=2& \raisebox{1.5ex}[-1.5ex] {$p=\infty$}
\end{array}
$$

\subsection{Gaussian $\rho$ and Exponential $\psi$}
Consider $D=\bbR$,
\[\rho(x)\,=\,\frac1{\sigma\,\sqrt{2\,\pi}}
\exp\left(\frac{-x^2}{2\,\sigma^2}\right)\quad\mbox{and}\quad
\psi(x)\,=\,\exp\left(\frac{-|x|}\lambda\right)
\]
for positive $\lambda$ and $\sigma$.
Now 
\begin{equation}\label{defomega1}
\omega(x)=\frac{\rho(x)}{\psi(x)}=\frac{1}{\sigma \sqrt{2 \pi}} \exp\left(-\frac{x^2}{2 \sigma^2}+\frac{|x|}{\lambda}\right),
\end{equation}
and
\begin{eqnarray*}
\|\omega^{1/\alpha}\|_{L_1(D)}^\alpha
&=&\frac1{\sigma\,\sqrt{2\,\pi}}\,\left(2\,\int_0^\infty
\exp\left(\frac{-x^2}{2\,\sigma^2\,\alpha}+\frac{x}{\lambda\,\alpha}\right)
\rd x\right)^\alpha\\
&=&\frac1{\sigma\,\sqrt{2\,\pi}}\,\left(2\,\int_0^\infty
\exp\left(\frac{-(x/\sigma-\sigma/\lambda)^2}{2\,\alpha}+
\frac{\sigma^2}{2\,\lambda^2\,\alpha}\right)\rd x\right)^\alpha\\
&=&\frac1{\sigma\,\sqrt{2\,\pi}}\,
\exp\left(\frac{\sigma^2}{2\,\lambda^2}\,\right)\,
\left(\sigma\,\sqrt{2\,\pi\,\alpha}\,\frac2{\sqrt{\pi}}\,
\int_{-\sigma/(\lambda\,\sqrt{2\,\alpha})}^\infty\exp(-y^2)\,\rd y\right)^\alpha\\
&=&\frac1{\sigma\,\sqrt{2\,\pi}}\,
\exp\left(\frac{\sigma^2}{2\,\lambda^2}\right)\,
\left(\sigma\,\sqrt{2\,\pi\,\alpha}\,\left(1+
{\rm erf}\left(\frac{\sigma}{\lambda\,\sqrt{2\,\alpha}}\right)\right)
\right)^\alpha.
\end{eqnarray*}
As before, we propose using $\kappa_a(x)=\exp(-|x|\,a)$. Hence
$\|\kappa_a^{1/\alpha}\|_{L_1}=2\,\alpha/a$ and the points $x_i$ are
given by \eqref{GG-quant}.

\subsubsection{Case of $p\le q$}
We have 
\[
\calE^q_p(\omega,\kappa_a)\,=\,
\left\|\frac{\omega}{\kappa_a}\right\|_{L_\infty(D)}\,=\,
\frac1{\sigma\,\sqrt{2\,\pi}}\,
\exp\left(\frac{\sigma^2\,(a+\lambda^{-1})^2}2\right).
\]

It is easy to verify that the minimum over $a>0$ satisfies
\[\min_{a>0}\|\kappa_a^{1/\alpha}\|_{L_1(D)}^\alpha\,
\left\|\frac\omega{\kappa_a}\right\|_{L_\infty(D)}
\,=\,\frac1{\sigma\,\sqrt{2\,\pi}}\,
\left(\frac{2\,\alpha}{a_*}\right)^\alpha\,
\exp\left(\frac{\sigma^2\,(a_*+\lambda^{-1})^2}2\right)
\]
for
\[
a_*\,=\,\frac{\sqrt{1+4\,\alpha\,\lambda^2/\sigma^2}-1}{2\,\lambda}.
\]

Therefore
\[{\rm FCTR}(p,q, \omega,\kappa_{a_*})\,=\,
\left(\sqrt{\frac{2\,\alpha}{\pi}}\,
\frac1{a_*\,\sigma\,\left(1+{\rm erf}(\sigma/\sqrt{2\,\alpha\,\lambda})\right)}
\right)^\alpha\, \exp\left(\frac{\sigma^2a_*(a_*+2/\lambda)}2\right).
\]
Note that the value of ${\rm FCTR}$ depends on $p$ and $q$ only via $\alpha$. 
Rounded values of ${\rm FCTR}$ for $\alpha \in \{1,2\}$ and $\sigma=1$ and 
various $\lambda$'s are\footnote{Computed with {\sc Mathematica}}:
\[\begin{array}{c||c|c|c|c|c|c||c} 
\lambda  & 1 & 5 & 10 & 20 & 30 & 100 & \\ \hline
{\rm FCTR} &   1.723 &  1.183  &  1.162  &  1.174  &  1.188  &  1.231 & \alpha=1\\
{\rm FCTR} & 2.468 &  1.460  &  1.436 &  1.465 &  1.491 &  1.573 & \alpha=2
\end{array}
\]

\subsubsection{Case of $p>q$}
We have 
\[\calE^q_p(\omega,\kappa_a)\,=\,\left(\frac{a}\alpha\right)^{\alpha-r}\,
\frac1{\sigma\,\sqrt{2\,\pi}}\,A^{\alpha-r},  
\]
where now
\begin{eqnarray*}
A&=&\int_0^\infty\exp\left(-\frac{x^2}{2\,\sigma^2\,(\alpha-r)}
+x\left(\frac{a}{\alpha-r}+\frac1{\lambda\,(\alpha-r)}-
\frac{a}\alpha\right)\right)\rd x\\
&=&\int_0^\infty\exp\left(-\frac1{2\,\sigma^2\,(\alpha-r)}\,
\left(x^2-2\,x\,\sigma^2\,\left(\frac{a\,r}\alpha+\frac1\lambda
\right)\right)\right)\rd x\\
&=&\exp\left(\frac{\sigma^2\,(\frac{a\,r}\alpha+\frac1\lambda)^2}
{2\,(\alpha-r)}\right)\,\int_0^\infty
\exp\left(-\frac{(x-\sigma^2(\frac{a\,r}\alpha+\frac1\lambda))^2}
{2\,\sigma^2\,(\alpha-r)}\right)\rd x\\
&=&\exp\left(\frac{\sigma^2\,(\frac{a\,r}\alpha+\frac1\lambda)^2}
{2\,(\alpha-r)}\right)\,\sqrt{\frac{\sigma^2\,\pi\,(\alpha-r)}{2}}\,
\left[1+{\rm erf}\left(\frac{\sigma\,(\frac{a\,r}\alpha+\frac1\lambda)}
  {\sqrt{2\,(\alpha-r)}}\right)\right].
\end{eqnarray*}
Hence 
\begin{eqnarray*}
\calE^q_p(\omega,\kappa)&=&\frac1{\sigma\,\sqrt{2\,\pi}}\,
\left(\frac{a^2\,\pi\,\sigma^2\,(\alpha-r)}{2\,\alpha^2}\right)^{(\alpha-r)/2}
\,\exp\left(\frac{\sigma^2}{2}\,\left(\frac{a\,r}\alpha+\frac1\lambda\right)^{2}\right)\\
&&\qquad\times
\left[1+{\rm erf}\left(\frac{\sigma\,(\frac{a\,r}\alpha+\frac1\lambda)}
  {\sqrt{2\,(\alpha-r)}}\right)\right]^{\alpha-r}.
\end{eqnarray*}

Since $$\frac{\|\kappa_a^{1/\alpha}\|_{L_1(D)}^{\alpha}}{\|\omega^{1/\alpha}\|_{L_1(D)}^{\alpha}}=
\left(\frac{2 \alpha}{a}\right)^{\alpha} \sigma \sqrt{2 \pi} \exp\left(-\frac{\sigma^2}{2 \lambda^2}\right) \left(\sigma \sqrt{2 \pi \alpha}\left(1+{\rm erf}\left(\frac{\sigma}{\lambda \sqrt{2 \alpha}}\right)\right)\right)^{-\alpha}$$ we obtain
\begin{eqnarray*}
\mbox{FCTR}(p,q,\omega,\kappa_{a}) & = & \left(\frac{1}{a\sigma} \sqrt{\frac{2 \alpha}{\pi}}\right)^{\alpha}   \left(\frac{a^2\,\pi\,\sigma^2\,(\alpha-r)}{2\,\alpha^2}\right)^{(\alpha-r)/2}
\,\exp\left(\frac{\sigma^2}{2}\,\left(\left(\frac{a\,r}\alpha+\frac1\lambda\right)^{2}-\frac{1}{\lambda^2}\right)\right)\\
&&\times \frac{\left[1+{\rm erf}\left(\frac{\sigma\,(\frac{a\,r}\alpha+\frac1\lambda)}{\sqrt{2\,(\alpha-r)}}\right)\right]^{\alpha-r}}{\left[1+{\rm erf}\left(\frac{\sigma}{\lambda \sqrt{2 \alpha}}\right)\right]^{\alpha}}.
\end{eqnarray*}

We again provide numerical results, first for the case $p=2$ and
$q=1$, i.e., $\alpha=r+1/2$.

For $r\in \{1,2\}$ and varying $a$, we obtain the following rounded values:
$$
\begin{array}{c||c|c|c|c||c|l}
a & 1 & 2 & 3 & 4 & & \\
\hline
{\rm FCTR}(2,1,\omega,\kappa_{a}) & 1.273   & 2.426 & 9.570 & 66.233 & \lambda=1,\ \sigma=1 & \\
{\rm FCTR}(2,1,\omega,\kappa_{a}) & 1.181   & 1.642 & 4.652 & 23.070 & \lambda=2,\ \sigma=1 & \raisebox{1.5ex}[-1.5ex] {$r=1$}\\
\hline
{\rm FCTR}(2,1,\omega,\kappa_{a}) &  1.747  & 2.546 & 12.473 & 146.677 & \lambda=1,\ \sigma=1 & \\
{\rm FCTR}(2,1,\omega,\kappa_{a}) &  1.747  & 1.729 & 5.683 & 44.797  & \lambda=2,\ \sigma=1 & \raisebox{1.5ex}[-1.5ex] {$r=2$}
\end{array}
$$

We now change $p$ to $p=\infty$, and choose again $q=1$, which implies $\alpha=r+1$.
For $r\in\{1,2\}$ and varying $a$ we obtain the following rounded values:

$$
\begin{array}{c||c|c|c|c||c|l}
a & 1 & 2 & 3 & 4 & & \\
\hline
{\rm FCTR}(\infty,1,\omega,\kappa_{a}) &  1.203  & 1.512  & 3.156  &  9.409 & \lambda=1,\ \sigma=1 & \\
{\rm FCTR}(\infty,1,\omega,\kappa_{a}) &  1.199   &  1.242  & 2.081  &  4.888 & \lambda=2,\ \sigma=1 & \raisebox{1.5ex}[-1.5ex] {$r=1$}\\
\hline
{\rm FCTR}(\infty,1,\omega,\kappa_{a}) & 1.724    & 1.700 & 4.509  &  23.434 & \lambda=1,\ \sigma=1 & \\
{\rm FCTR}(\infty,1,\omega,\kappa_{a}) &  1.827  & 1.366  & 2.647 &  9.897 & \lambda=2,\ \sigma=1 & \raisebox{1.5ex}[-1.5ex] {$r=2$}
\end{array}
$$

\subsection{Log-Normal $\rho$ and constant $\psi$}
Consider $D=\bbR_+$,  $\psi(x)=1$ and 
\begin{equation}\label{rholognor}
\rho(x)\,=\,\omega(x)\,=\,\frac1{x\,\sigma\,\sqrt{2\,\pi}}\,
\exp\left(-\frac{(\ln x-\mu)^2}{2\,\sigma^2}\right)
\end{equation}
for given $\mu\in\bbR$ and $\sigma>0$. 

For $\kappa$ we take
\[
\kappa_c(x)\,=\,\left\{\begin{array}{ll}
1  &  \mbox{if\ }x\in[0,{\rm e}^{\mu}],\\
\exp(c\,(\mu-\ln x))  &  \mbox{if\ }x >{\rm e}^{\mu},\end{array}\right.
\]
for positive $c$. For $\kappa_c^{1/\alpha}$ to be integrable we have to restrict $c$
so that
\[
c>\alpha.
\]

It can be checked that
\[\|\kappa_c^{1/\alpha}\|^\alpha_{L_1(D)}\,=\,\left(\frac{c}{c-\alpha}
\right)^\alpha\,{\rm e}^{\alpha\,\mu}.
\]
Then the points $x_i$ for $i=0,1,\ldots,n$ that satisfy \eqref{knots1} are given by 
$$x_i=\left\{
\begin{array}{ll}
 \frac{c}{c-\alpha}\, {\rm e}^{\mu} \,\frac{i}{n} & \mbox{ for $i \le n \, \frac{c-\alpha}{c}$,}\\
 {\rm e}^{\mu} \, \left(\frac{\alpha}{c}\, \frac{n}{n-i}\right)^{\alpha/(c-\alpha)} & \mbox{ otherwise.}
\end{array}\right.$$

\subsubsection{Case of $p \le q$}

We determine $\|\omega/\kappa_c\|_{L_\infty(D)}$.
For $x\le{\rm e}^{\mu}$ we have 
\[
\frac{\omega(x)}{\kappa_c(x)}\,=\,\omega(x)\,=\,
\frac1{\sigma\,\sqrt{2\,\pi}}\,
\exp\left(-\frac{(t-\mu)^2}{2\,\sigma^2}-t\right)\quad
\mbox{with}\quad t\,=\,\ln x\le\,\mu.
\]
Its maximum is attained at $t=\mu-\sigma^2$ and
\[\max_{x\le{\rm e}^{\mu}}\frac{\omega(x)}{\kappa_c(x)}\,=\,
\frac1{\sigma\,\sqrt{2\,\pi}}\,\exp\left(\frac{\sigma^2}2-\mu\right). 
\]
For $x>{\rm e}^{\mu}$,
\[
\frac{\omega(x)}{\kappa_c(x)}\,=\,
\frac1{\exp(c\, \mu)\,\sigma\,\sqrt{2\,\pi}}\,\exp\left(-\frac{(t-\mu)^2}{2\,\sigma^2}+
t\,(c-1)\right)\quad\mbox{with $t\,=\,\ln x\,>\,\mu$.}
\]
The maximum of the expression above is attained at $t=\mu+\sigma^2\,(c-1)$ and
\begin{eqnarray*}
\sup_{x>{\rm e}^{\mu}}\frac{\omega(x)}{\kappa_c(x)}\,&=&\,
\frac1{\exp(c\, \mu)\,\sigma\,\sqrt{2\,\pi}}\,\exp\left((c-1)\,\mu
+\frac{(c-1)^2\,\sigma^2}2\right)\\
\,&=&\,
\frac1{\sigma\,\sqrt{2\,\pi}}\,\exp\left(-\mu+
\frac{(c-1)^2\,\sigma^2}2\right).
\end{eqnarray*}
This yields that
\[
\left\|\frac{\omega}{\kappa_c}\right\|_{L_\infty(D)}\,=\,\frac1{\sigma\,
  \sqrt{2\,\pi}}\,\exp\left(-\mu+\frac{\sigma^2}2\,\max(1,(c-1)^2)
\right).
\]

To find the optimal value of $c$, note that 
\[
\left\|\frac\omega{\kappa_c}\right\|_{L_\infty(D)}\,
\|\kappa_c^{1/\alpha}\|_{L_1(D)}^\alpha\,=\,
\frac{{\rm e}^{(\alpha-1)\,\mu}}{\sigma\,\sqrt{2\,\pi}}\,
\left(f(c)\right)^\alpha,
\]
where $f(c)$ is given by 
\[
f(c)\,=\,\exp\left(\frac{\sigma^2\,\max(1,(c-1)^2)}{2\,\alpha}\right)\,
\left(1+\frac\alpha{c-\alpha}\right). 
\]

Consider first $\alpha\ge 2$ and recall the restriction $c>\alpha$.
For such values of $c$ we have $$f(c)\,=\,\exp\left(\frac{\sigma^2\,(c-1)^2}{2\,\alpha}\right)\,
\left(1+\frac\alpha{c-\alpha}\right)$$ and hence
\begin{eqnarray*}
  f'(c)&=&   \frac{\sigma^2}{\alpha\,(c-\alpha)^2}\,
  \exp\left(\frac{\sigma^2}{2\alpha}\,(c-1)^2\right)\,
  \left(c\,(c-1)\,(c-\alpha)-\frac{\alpha^2}{\sigma^2}\right).  
\end{eqnarray*}
Therefore, 
\[
\min_{c>\alpha} f(c)\,=\,f(c_*)\,=\,\exp\left(
\frac{\sigma^2\,(c_*-1)^2}{2\,\alpha}\right)\,\frac{c_*}{c_*-\alpha}
\]
for $c_*$ such that
\begin{equation}\label{c-star}
c_*\,>\,\alpha\quad\mbox{and}\quad c_*\,(c_*-1)\,(c_*-\alpha)\,=\,
\frac{\alpha^2}{\sigma^2}. 
\end{equation}

Consider next $\alpha\in (0,2)$. Then for $c\le2$, the minimum of
$f(c)$ is attained in $c=2$, and it is a global minimum if
$2(2-\alpha)\ge\alpha^2/\sigma^2$. Otherwise, the minimum is at
$c_*$ given by \eqref{c-star}.

In summary, for $\alpha > 0$, we have
\[
\min_{c>\alpha}\left\|\frac\omega{\kappa_c}\right\|_{L_\infty(D)}\,
\|\kappa_c^{1/\alpha}\|_{L_1(D)}^\alpha\,=\,
\frac{{\rm e}^{(\alpha-1)\,\mu}}{\sigma\,\sqrt{2\,\pi}}\times
\begin{cases}
\exp\left(\frac{\sigma^2\,(c_*-1)^2}2\right)\,
\left(\frac{c_*}{c_*-\alpha}\right)^\alpha &
\mbox{if $\alpha\, \ge \, 2$}\\
&\mbox{or $2\,(2-\alpha)\,  \le \, \frac{\alpha^2}{\sigma^2}$},\\
\exp\left(\frac{\sigma^2}2\right)\,\left(\frac2{2-\alpha}\right)^\alpha
&\mbox{otherwise}.  
\end{cases}
\]

To derive the value of the $L_1$ norm of $\omega^{1/\alpha}$, we 
will use the following well-known facts: If ${\bf X}_{\sigma,\mu}$ is a log-normally distributed random 
variable with parameters $\sigma$ and $\mu$, then the mean value and the variance of 
${\bf X}_{\sigma,\mu}$ are, respectively, equal to 
\[
\expect({\bf X}_{\sigma,\mu})\,=\,\exp\left(\sigma^2/2+\mu\right)
\quad\mbox{and}\quad\expect\left({\bf X}_{\sigma,\mu}
-\expect({\bf X}_{\sigma,\mu})\right)^2\,=\,
\left(\exp\left(\sigma^2\right)-1\right)\,\exp\left(\sigma^2+2\,
\mu\right). 
\]
Hence
\begin{equation}\label{well-known}
\expect\left({\bf X}_{\sigma,\mu}^2\right)\,=\,
\exp\left(2\sigma^2+2\,\mu\right). 
\end{equation}

If $\alpha=1$, then $\|\omega^{1/\alpha}\|_{L_1(D)}^\alpha=1$, and then 
\[
\mbox{FCTR}(p,q,\omega,\kappa_{c_*})\,=\,
\frac{1}{\sigma\,\sqrt{2\,\pi}}\,
\begin{cases}
\frac{c_*}{c_*-1}\,\exp\left(\frac{\sigma^2\,(c_*-1)^2}2\right)
&\mbox{if $2\, \le \, \frac{1}{\sigma^2}$},\\[0.5em]
2\, \exp\left(\frac{\sigma^2}2\right)
&\mbox{otherwise}.  
\end{cases}
\]

For $\alpha\in (1,2)$, to simplify the notation, we will use, in the following, parameters $s$
and $\gamma$ given by 
\[
s\,=\,\frac{2\,\alpha}{\alpha-1}\quad\mbox{and}\quad
\gamma\,=\,\frac{\sigma\,\sqrt{\alpha}}{s}.
\]
The change of the variable $x=t^s$ gives 
\begin{eqnarray*}
(\sigma\,\sqrt{2\,\pi})^{1/\alpha}\,
\|\omega^{1/\alpha}\|_{L_1(D)} &=&
\int_0^\infty\frac1{x^{1/\alpha}}\,
\exp\left(\frac{-(\ln x-\mu)^2}{2\,\alpha\,\sigma^2}
\right)\rd x\\
&=& s\,\int_0^\infty t^{s-s/\alpha-1}\,
\exp\left(\frac{-(\ln t^s-\mu)^2}{2\,\alpha\,\sigma^2}\right)\rd t\\
&=& s\,\int_0^\infty t\,\exp\left(\frac{-(\ln t-\mu/s)^2}{2\,
(\sigma\,\sqrt{\alpha}/s)^2}\right)
\rd t\\   &=&
 s\,\gamma\,\sqrt{2\,\pi}\,\int_0^\infty\frac{t^2}{t\,\gamma
\,\sqrt{2\,\pi}}
\,\exp\left(\frac{-(\ln t-\mu/s)^2}{2\,\gamma^2}\right)\rd t. 
\end{eqnarray*}
The last integral is the expected value of the square 
of a log-normal random variable ${\bf X}_{\gamma,\mu/s}$ with the
parameter $\sigma$ replaced by $\gamma$ and $\mu$ replaced by $\mu/s$.
Hence
$$
\|\omega^{1/\alpha}\|_{L_1(D)}^\alpha =
\frac{(s\,\gamma\,\sqrt{2\,\pi})^\alpha}{\sigma\,\sqrt{2\,\pi}}
\,
\exp\left(2\,\gamma^2\,\alpha+\frac{2\,\mu\,\alpha}{s}\right)
\, = \, \frac{(\sigma\,\,\sqrt{2\,\pi\,\alpha})^\alpha}{\sigma\,\sqrt{2\,\pi}}\,
\exp\left(\frac{\sigma^2\,(\alpha-1)^2}{2}+\mu\,(\alpha-1)\right).
$$

This gives us
\[
{\rm FCTR}(p,q,\omega,\kappa_{c_*})\,=\,
\left(\frac{c_*}{(c_*-\alpha)\,\sigma\,\sqrt{2\,\pi\,\alpha}}
\right)^\alpha\,\exp\left(\frac{\sigma^2\,((c_*-1)^2-(\alpha-1)^2)}2
\right)
\]
if either $\alpha\ge2$ or $\alpha<2$ and $2(2-\alpha)\le\alpha^2/\sigma^2$,
and
\[
{\rm FCTR}(p,q,\omega,\kappa_{2})\,=\,
\left(\frac{2}{(2-\alpha)\,\sigma\,\sqrt{2\,\pi\,\alpha}}
\right)^\alpha\,\exp\left(\frac{\sigma^2\,(1-(\alpha-1)^2)}2\right)
\]
if $\alpha<2$ and $2(2-\alpha)>\alpha^2/\sigma^2$.

Rounded values for FCTR for various $\sigma$ and $\alpha$ are\footnote{Computed with {\sc Mathematica}.}:
\[\begin{array}{c||c|c|c||c}
\sigma & 1 & 2 & 3 & \ \\
\hline
{\rm FCTR} & 1.315 & 2.948 & 23.941 & \alpha=1\\
\hline
{\rm FCTR} & 2.988 & 4.615 & 7.573 & \alpha=2
\end{array}
\]

\subsubsection{Case of $p > q$}

Now 
\[\calE^q_p(\omega,\kappa_c)\,=\,\frac1{\sigma\,\sqrt{2\,\pi}}\,
\left(\frac{c-\alpha}{c\,{\rm e}^{\mu}}\right)^{\alpha-r}\,
(I_1+I_2)^{\alpha-r},
\]
where 
\[I_1\,=\int_0^{{\rm e}^{\mu}} \exp\left(-\frac1{\alpha-r}\,\left[
  \frac{(\ln x-\mu)^2}{2\,\sigma^2}+\ln x\right]\right)\rd x
\]
and
\[I_2\,=\,\int_{{\rm e}^{\mu}}^\infty \exp\left(-\frac1{\alpha-r}\,\left[
  \frac{(\ln x-\mu)^2}{2\,\sigma^2}+\ln x\right]
-\frac{r\,c}{\alpha\,(\alpha-r)}\,(\mu-\ln x)\right)\rd x.
\]
In what follows, for both integrals, we will use first the change of
variables $y=\ln x-\mu$. We have
\begin{eqnarray*}
  I_1&=&\int_{-\infty}^0 \exp(y+\mu)\,\exp\left(-\frac1{\alpha-r}\,
  \left[\frac{y^2}{2\,\sigma^2}+y+\mu\right]\right)\rd x\\
  &=&\exp\left(\mu\,\frac{\alpha-r-1}{\alpha-r}\right)\,
  \int_{-\infty}^0 \exp\left(-\frac1{\alpha-r}\,\left[
    \frac{y^2}{2\,\sigma^2}+(1+r-\alpha)\,y\right]\right)\rd x\\
&=&\exp\left(\mu\,\frac{\alpha-r-1}{\alpha-r}\right)\,
  \int_{-\infty}^0 \exp\left(-\frac{
    y^2+2\,y\,\sigma^2\,(1+r-\alpha)}{2\sigma^2\,(\alpha-r)}\right)\rd x\\
&=&\exp\left(\frac{1+r-\alpha}{\alpha-r}\,\left(\frac{\sigma^2\,(1+r-\alpha)}2-\mu\right)\right)\,
 \int_{-\infty}^0\exp\left(-\frac{[y+\sigma^2\,(1+r-\alpha)]^2}
 {(\alpha-r)\,2\,\sigma^2}  \right)\rd y\\
&=&\exp\left(\frac{1+r-\alpha}{\alpha-r}\,\left(\frac{\sigma^2\,(1+r-\alpha)}2
-\mu\right)\right)
\sqrt{\frac{\sigma^2\,(\alpha-r)\,\pi}{2}}\,
\left[1+{\rm erf}\left(\frac{\sigma\,(1+r-\alpha)}{\sqrt{2\,(\alpha-r)}}\right)\right].
\end{eqnarray*}
Similarly for $I_2$ we get
\begin{eqnarray*}
I_2&=&\exp\left(\mu\,\frac{\alpha-r-1}{\alpha-r}\right)\,
\int_0^\infty\exp\left(-\frac1{\alpha-r}\left[\frac{y^2}{2\,\sigma^2}+y
  -y\,\left(\alpha-r+\frac{r\,c}\alpha\right)\right]\right) \rd y\\
&=&\frac{\exp\left(\frac{\sigma^2\,(1+r-\alpha-r\,c/\alpha)^2}
        {2\,(\alpha-r)}\right)}
{\exp\left(\frac{1+r-\alpha}{\alpha-r}\,\mu\right)}\,\int_0^\infty
\exp\left(-\frac{[y+\sigma^2\,(1+r-\alpha-r\,c/\alpha)]^2}{(\alpha-r)\,2\,
  \sigma^2}\right)\rd y\\
&=&\frac{\exp\left(\frac{\sigma^2\,(1+r-\alpha-r\,c/\alpha)^2}
        {2\,(\alpha-r)}\right)}
{\exp\left(\frac{1+r-\alpha}{\alpha-r}\,\mu\right)}\,
\sqrt{\frac{\sigma^2\,(\alpha-r)\pi}{2}}\,
\left[1-{\rm erf}\left(\frac{\sigma\,(1+r-\alpha-r\,c/\alpha)}
  {\sqrt{2\,(\alpha-r)}}\right)\right].
\end{eqnarray*}
Hence
\begin{eqnarray*}
\lefteqn{(I_1+I_2)^{\alpha-r}}\\
&=& \frac{\exp(\sigma^2\,(1+r-\alpha)^2/2)}
                            {\exp(\mu\,(1+r-\alpha))}  \,
\left[\frac{\sigma^2\,(\alpha-r)\,\pi}2\right]^{(\alpha-r)/2}\,
\left[1+{\rm erf}\left(\frac{\sigma\,(1+r-\alpha)}
                              {\sqrt{2\,(\alpha-r)}}\right)\right.\\
&&\left.+\exp\left(\frac{\sigma^2}{2\,(\alpha-r)}\,\left(
-\frac{2\,r\,c}\alpha\,(1+r-\alpha)+\left(\frac{r\,c}\alpha\right)^2\right)\right)
\,\left[1-{\rm erf}\left(\frac{\sigma\,(1+r-\alpha-r\,c/\alpha)}
           {\sqrt{2\,(\alpha-r)}}\right)\right]\right]^{\alpha-r}.
\end{eqnarray*}

Since computing ${\rm FCTR}(p,q,\omega,\kappa_c)$ for arbitrary
parameters $q\le p$ is very challenging, we will do this for 
$p=\infty$ and $q=1$, which---as already mentioned---corresponds to the
integration problem.
In this specific case, we have $\alpha=r+1$ and
\[
(I_1+I_2)^{\alpha-r}\,=\,\sqrt{\frac{\sigma^2\,\pi}2}\,
\left[1+\exp\left(\frac{(\sigma\, (\alpha-1) \,c)^2}{2\,\alpha^2}\right)\,
\left[1-{\rm erf}\left(-\frac{\sigma\,(\alpha-1)\,c}{\alpha\,\sqrt{2}}\right)\right]
 \right].
\]
This yields
\begin{eqnarray*}
{\rm FCTR}(\infty,1,\omega,\kappa_c)&=&
\frac{(c-\alpha)\,\sigma\,\sqrt{2\,\pi}}{2\,c}\,
\left(\frac{c}{(c-\alpha)\,\sigma\,\sqrt{2\,\pi\,\alpha}}\right)^\alpha\,
\exp\left(-\frac{\sigma^2\,(\alpha-1)^2}2 -\mu(\alpha-1)\right)\\
&&\qquad\qquad\times
\left[1+\exp\left(\frac{(\sigma\,(\alpha-1)\,c)^2}{2\,\alpha^2}\right)\,
\left[1-{\rm erf}\left(\frac{-\sigma\,(\alpha-1)\,c}{\alpha\,\sqrt{2}}\right)\right]
\right].                               
\end{eqnarray*}

As a numerical example we consider the case $\mu=0$ and $\sigma=1$. 
For fixed $\alpha \in \{1.5,2,2.5,3,3.5\}$ we numerically minimize\footnote{Using the {\sc Mathematica} command \texttt{FindMinimum}}  ${\rm FCTR}(\infty,1,\omega,\kappa_c)$ as a function in $c$. The results together with the optimal $c_*$ are presented in the following table:  
\[\begin{array}{c||c|c|c|c|c}
  \alpha & 1.5 & 2 & 2.5 & 3 & 3.5\\
  \hline
  {\rm FCTR}(\infty,1,\omega,\kappa_{c_*}) & 1.058 & 1.224 & 1.594 & 2.314 & 3.648\\
  c_* & 2.555 & 2.973 & 3.422 & 3.899 & 4.392
\end{array}
\]

\subsection{Logistic $\rho$ and Exponential $\psi$}
Consider $D=\bbR$, 
\[
\rho(x)\,=\frac{\exp(x/\nu)}{\nu\,(1+\exp(x/\nu))^2}\quad\mbox{and}\quad
\psi(x)\,=\, \exp(-b|x|)
\]
with parameters $\nu>0$ and $b>0$. Then 
\[
\omega(x)\,=\,\frac{\exp(x/\nu+b\,|x|)}{\nu\,(1+\exp(x/\nu))^2}
\]
which is quite complicated, in particular if one considers $\omega^{1/\alpha}$, and is not 
monotonic. 
Consider therefore 
\[
\kappa_a(x)\,=\,\exp(-a|x|)\quad\mbox{for some $a\,>\,0$}.
\]
Hence the points $x_{-n},\dots,x_{n}$ satisfying \eqref{condxigen}
are again given by \eqref{GG-quant}.

To simplify the formulas to come, we use
\[
\lambda\,:=\,\frac1\nu,\quad\mbox{i.e.,}\quad
\omega(x)\,=\,\frac{\lambda\,\exp(\lambda\,x+b\,|x|)}{(1+\exp(\lambda\,x))^2}.
\]
For $\|\omega^{1/\alpha}\|_{L_1(D)}^{\alpha}$ and $\|\omega/\kappa_a\|_{L_\infty(D)}$ 
to be finite, we need to have
\[
\lambda\,>\,b \quad\mbox{and}\quad \lambda\,\ge\,a+b.
\]

Since the integral in $\calE^q_p(\omega,\kappa_a)$ becomes very complicated for this example we 
do not distinguish between $p \le q$ and $p>q$. Instead we use the upper bound \eqref{fctr} here.

We first study $\|\omega/\kappa_a\|_{L_\infty(D)}$. 
Since $\omega$ and $\kappa_a$ are symmetric, we can restrict the
attention to $x\ge0$. 
By substituting $z=\exp(\lambda\,x)$, we get that
$$\left\|\frac{\omega}{\kappa_a}\right\|_{L_\infty(D)}=\lambda\sup_{z\ge1} \frac{z^{1+(a+b)/\lambda}}{(1+z)^2}.$$
When $a+b=\lambda$ the supremum is attained at $z=\infty$, otherwise
it is attained at $z=(\lambda+a+b)/(\lambda-(a+b))$.
Therefore
\[
\left\|\frac{\omega}{\kappa_a}\right\|_{L_\infty(D)}\,=\,\frac\lambda4\,
\left(1+\frac{a+b}\lambda\right)^{1+(a+b)/\lambda}\,
\left(1-\frac{a+b}\lambda\right)^{1-(a+b)/\lambda},
\]
with the convention that $0^0:=1$, i.e.,
$\|\omega/\kappa_a\|_{L_\infty(D)}=\lambda$ if $a=\lambda-b$.

Indeed, the previous formula for $\left\|\omega/\kappa_a\right\|_{L_\infty(D)}$ 
can be shown by noting that
\begin{eqnarray*}
 \lefteqn{\lambda\, \left[\frac{\lambda + a + b}{\lambda - a -b}\right]^{1+\frac{a+b}{\lambda}}\, 
 \left(1+\frac{\lambda + a + b}{\lambda - a -b}\right)^{-2}
 =\lambda \, \left[\frac{\lambda + a + b}{\lambda - a -b}\right]^{1+\frac{a+b}{\lambda}}\, 
 \left(\frac{\lambda-(a+b)}{2\lambda}\right)^2}\\
 &=&\frac{\lambda}{4}\, \left[\frac{\lambda + a + b}{\lambda - a -b}\right]^{1+\frac{a+b}{\lambda}}
 \left(1-\frac{a+b}{\lambda}\right)^2\\
 &=&\frac{\lambda}{4}\, \left[\frac{\lambda + a + b}{\lambda - a -b}\right]^{1+\frac{a+b}{\lambda}}
 \left(1-\frac{a+b}{\lambda}\right)^{1-\frac{a+b}{\lambda}}\, \left(1-\frac{a+b}{\lambda}\right)^{1+\frac{a+b}{\lambda}}\\
 &=&\frac{\lambda}{4}\, \left(1-\frac{a+b}{\lambda}\right)^{1-\frac{a+b}{\lambda}}\, 
 \left(\frac{\lambda + a + b}{\lambda - a -b}\cdot \frac{\lambda-a-b}{\lambda}\right)^{1+\frac{a+b}{\lambda}}.
\end{eqnarray*}

As above,
\[
 \|\kappa_a^{1/\alpha}\|_{L_1(D)}^\alpha\,=\,
\left(\frac{2\alpha}{a}\right)^\alpha. 
\]
We also have
\begin{eqnarray*}
\|\omega^{1/\alpha}\|_{L_1(D)}^\alpha&=&\lambda\,
  \left(2\,\int_0^\infty\frac{\exp((\lambda+b)\,x/\alpha)}
       {(1+\exp(\lambda\,x))^{2/\alpha}}\, \rd x\right)^\alpha\\
&\ge&\lambda\,\left(2\int_0^\infty
\frac{\exp(\lambda\,x/\alpha)}
{(1+\exp(\lambda\,x/\alpha))^2}\, \rd
    x\right)^\alpha
\end{eqnarray*}
due to the fact that $1/(1+A)^{1/\alpha}\ge 1/(1+A^{1/\alpha})$ since
$\alpha\ge1$.
Therefore
\[
\|\omega^{1/\alpha}\|_{L_1(D)}^\alpha\,\ge\,\lambda\,
\left(\frac\alpha\lambda\right)^\alpha.
\]
This gives 
\[
  {\rm FCTR}(p,q,\omega,\kappa_a)\,\le\ 
  \left(\frac{2\,\lambda}{a}\right)^\alpha\,\frac14\,
  \left(1+\frac{a+b}\lambda\right)^{1+(a+b)/\lambda}\,
  \left(1-\frac{a+b}\lambda\right)^{1-(a+b)/\lambda}.
\]
As before the right-hand side above is
\[
  \left(\frac{2\,\lambda}{\lambda-b}\right)^\alpha\quad\mbox{if\ }
  a\,=\,\lambda-b.
\]

Letting $x=a/\lambda$, the minimum is at $0 < x < 1-b/\lambda$ that is
the root of
\[
x\,\left(\ln\left(1+\frac{b}{\lambda}+x\right)-\ln\left(1-\frac{b}{\lambda}-x\right)\right)-\alpha\,=\,0. 
\]

Rounded values of the upper bound on FCTR for $\alpha=b=1$ and various $\lambda$'s are\footnote{Computed with {\sc Mathematica}.}:
$$
\begin{array}{r||c|c|c|c}
\lambda & 2 & 5 & 10 & 15\\
\hline
{\mbox{Bound on FCTR}} & 3.341 & 1.710 & 1.431 & 1.353
\end{array}
$$

\subsection{Student's $\rho$ and $\psi$}
Consider Student's $t$-distribution on $D=\bbR$ 
\[
\rho(x)\,=\,T_\nu\,\left(1+\frac{x^2}\nu\right)^{-(\nu+1)/2}\quad
\mbox{with}\quad T_\nu\,=\,\frac{\Gamma((\nu+1)/2)}{\sqrt{\nu\,\pi}\,
  \Gamma(\nu/2)}\quad\mbox{for}\quad\nu\,>\,0.
\]
Here $\Gamma$ denotes Euler's Gamma function $\Gamma(z)=\int_0^{\infty} t^{z-1} {\rm e}^{-t}\, \rd t$. Let 
\[\psi(x)\,=\,\left(1+\frac{x^2}\nu\right)^{-b/2}\quad\mbox{and}\quad
\kappa_a(x)\,=\,(1+|x|)^{-a}
\]
for $a>0$ and $b\ge 0$. For $\|\omega^{1/\alpha}\|_{L_1(D)}$,
$\|\kappa_a^{1/\alpha}\|_{L_1(D)}$, and
$\|\omega/\kappa_a\|_{L_\infty(D)}$ to be finite, we have to assume that
\[\nu+1-b\,\ge\,a\,>\,\alpha. 
\]
It is easy to see that
\[\|\kappa_a^{1/\alpha}\|_{L_1(D)}^{\alpha}\,=\,\left(\frac{2\,\alpha}{a-\alpha}
\right)^\alpha.
\]
Hence the points $x_{-n},\dots,x_{n}$ satisfying \eqref{condxigen}
are given by
\begin{equation*}
    x_{i}\,=\,-x_{-i}\,=\, \left(1-\frac{i}{n}\right)^{-\frac{\alpha}{a-\alpha}} -1 \quad\mbox{for\ } 0 \le i \le n.
\end{equation*}

To compute the norm of $\omega^{1/\alpha}$, 
make the change of variables $x/\sqrt{\nu}=t/\sqrt{\mu}$, where
\[\mu\,=\,\frac{\nu+1-b-\alpha}\alpha
  \quad\mbox{so that}\quad
  \frac{\mu+1}2\,=\,\frac{\nu+1-b}{2\,\alpha}.
\]
Then we get 
\begin{eqnarray*}
\|\omega^{1/\alpha}\|_{L_1(D)}^\alpha&=&T_\nu\,\left(\int_\bbR\left(1+
\frac{x^2}\nu\right)^{-(\nu+1-b)/(2\,\alpha)}\rd x\right)^\alpha\\
&=&T_\nu\left(\frac\nu\mu\right)^{\alpha/2} T_\mu^{-\alpha}\,
\left(T_\mu\int_\bbR\left(1+\frac{t^2}\mu\right)^{-(\mu+1)/2}\rd t
\right)^\alpha\,=\,T_\nu\,\left(\frac{\sqrt{\nu}}{T_\mu\,\sqrt{\mu}}
\right)^\alpha.
\end{eqnarray*}
Since
\[\frac{\omega(x)}{\kappa_a(x)}\,=\,
T_\nu\,\left(1+\frac{x^2}\nu\right)^{-(\nu+1-b)/2}\,
(1+|x|)^a,
\]
we have 
\[
\left\|\frac\omega{\kappa_a}\right\|_{L_\infty(D)}\,=\,
T_\nu\,(1+\nu)^{(\nu+1-b)/2} \quad\mbox{for\ }a\,=\,\nu+1-b,
\]
and
\[\left\|\frac\omega{\kappa_a}\right\|_{L_\infty(D)}\,=\,
\frac{\omega(x_*)}{\kappa(x_*)}\quad\mbox{for}\quad
x_*\,=\,\frac{\sqrt{(\nu+1-b)^2 +4\,a\,\nu\,(\nu+1-b-a)}
  -(\nu+1-b)}{2\,(\nu+1-a-b)}
\]
for $a<\nu+1-b$. 

This gives
\begin{eqnarray*}
{\rm FCTR}(p,q,\omega,\kappa_a)\,\le\,\left\{\begin{array}{ll}
(1+\nu)^{(\nu+1-b)/2}\,\left(\frac{2\,T_\mu}{\sqrt{\nu\,\mu}}\right)^\alpha
& \mbox{for $a=\nu+1-b$},\\[0.5em]
\frac{(1+x_*)^a}{\left(1+\frac{x_*^2}{\nu}\right)^{(\nu+1-b)/2}}\,
\left(T_\mu\,\frac{2\,\alpha}{a-\alpha}\,\sqrt{\frac\mu\nu}\right)^\alpha
& \mbox{for $a\in(\alpha,\nu+1-b)$},\end{array}\right. 
\end{eqnarray*}
with equality whenever $p \le q$.

In the following numerical experiments for fixed values of $\alpha$, $b$ and $\nu$, 
we choose $a \in (\alpha,\nu+1-b]$ of the form $a=\alpha+k/10$ such that it gives the smallest value of the above bound on FCTR. For example: 
$$
\begin{array}{c||c|c|c|c}
(\nu,b,\alpha) & (3,2,1) & (4,2,2) & (5,3,2) & (6,3,3)\\
\hline
{\rm FCTR} & 1.427 & 1.626 & 1.710 & 1.861
\end{array}
$$

\begin{small}
\noindent\textbf{Authors' addresses:}

\medskip
\noindent Peter Kritzer\\
Johann Radon Institute for Computational and Applied Mathematics (RICAM)\\
Austrian Academy of Sciences\\
Altenbergerstr.~69, 4040 Linz, Austria\\
E-mail: \texttt{peter.kritzer@oeaw.ac.at}

\medskip

\noindent Friedrich Pillichshammer\\
Institut f\"{u}r Finanzmathematik und Angewandte Zahlentheorie\\
Johannes Kepler Universit\"{a}t Linz\\
Altenbergerstr.~69, 4040 Linz, Austria\\
E-mail: \texttt{friedrich.pillichshammer@jku.at}

\medskip

\noindent Leszek Plaskota\\
Institute of Applied Mathematics and Mechanics\\
Faculty of Mathematics, Informatics, and Mechanics\\
University of Warsaw\\
S. Banacha 2, 02-097 Warsaw, Poland\\
E-mail: \texttt{leszekp@mimuw.edu.pl} 

\medskip

\noindent G. W. Wasilkowski\\
Computer Science Department, University of Kentucky\\
301 David Marksbury Building\\
329 Rose Street\\
Lexington, KY 40506, USA\\
E-mail: \texttt{greg@cs.uky.edu}
\end{small}

\end{document}